\documentclass{amsart}
\numberwithin{equation}{section}

\newtheorem{thm}{Theorem}[section]
\newtheorem{lem}[thm]{Lemma}

\newtheorem{defin}[thm]{Definition}

\begin{document}
\title{Inverse problem of determining an order of the
Riemann-Liouville time-fractional derivative}

\author{Shavkat Alimov}
\author{Ravshan Ashurov}
\address{National University of Uzbekistan named after Mirzo Ulugbek and Institute of Mathematics, Uzbekistan Academy of Science}
\curraddr{Institute of Mathematics, Uzbekistan Academy of Science,
Tashkent, 81 Mirzo Ulugbek str. 100170} \email{ashurovr@gmail.com}

\small

\title[Inverse problem of determining an order of time-fractional
derivative] {Inverse problem of determining an order of the
Riemann-Liouville time-fractional derivative}

\begin{abstract}

The inverse problem of determining the order of the fractional   Riemann-Liouville derivative with respect to  time in the subdiffusion
equation with an arbitrary positive self-adjoint operator having a discrete spectrum is considered. Using the classical Fourier
method it is proved, that the value of the norm $||u(t)||$ of the
solution at a fixed time instance recovers uniquely the order of
derivative. A list of examples is discussed, including a linear
system of fractional differential equations, differential models
with involution, fractional Sturm-Liouville operators, and many
others.

\vskip 0.3cm \noindent {\it AMS 2000 Mathematics Subject
Classifications} :
Primary 35R11; Secondary 74S25.\\
{\it Key words}: Determination of order of derivative, Fourier method, inverse and initial-boundary value problem, Riemann-Liouville
derivatives, subdiffusion equation. 
\end{abstract}

\maketitle

\section{Main result}

It has long been known that to model subdiffusion (anomalous or
slow diffusion) processes, it is necessary to use differential
equations of fractional order $\rho\in (0,1)$. But in this case,
unlike differential equations of integer order, an order of the
fractional derivative $\rho$ is often unknown and difficult to be
directly measured. The determination of this
parameter is called the inverse problem of determining the order of the fractional derivative. These inverse problems are
not only theoretically interesting, but also necessary for finding solutions to initial-boundary value problems and studying
properties of solutions. The paper \cite{ash1}  by Li, Liu, Yamamoto surveys works on such inverse problems.

In the present paper, we are concerned with inversion for order in the subdiffusion equation with the Riemann-Liouville
time-fractional derivative. 

Let $H$ be a separable Hilbert space with the scalar product
$(\cdot, \cdot)$ and the norm $||\cdot||$ and  $A: H\rightarrow H$
be an arbitrary positive selfadjoint operator in $H$. Suppose that
$A$ has a complete in $H$ system of orthonormal eigenfunctions
$\{\upsilon_k\}$ and a countable set of nonnegative eigenvalues
$\lambda_k$. It is convenient to assume that the eigenvalues do
not decrease as their number increases, i.e.
$0<\lambda_1\leq\lambda_2 \cdot\cdot\cdot$.

Using the definitions of a strong integral and a strong
derivative, fractional analogues of integrals and derivatives can
be determined for vector-valued functions (or simply functions)
$h: \mathbb{R}_+\rightarrow H$, while the well-known formulae and
properties are preserved (see, for example, \cite{ash2}).Recall that the fractional integration of order $\rho<0$ of the function
$h(t)$ defined on $[0, \infty)$ has the form
$$
\partial_t^\rho h(t)=\frac{1}{\Gamma
(-\rho)}\int\limits_0^t\frac{h(\xi)}{(t-\xi)^{\rho+1}} d\xi, \quad
t>0,
$$
provided the right-hand side exists. Here $\Gamma(\rho)$ is
Euler's gamma function. Using this definition one can define the
Riemann - Liouville fractional derivative of order $\rho$,
$0<\rho< 1$, as
$$
\partial_t^\rho h(t)= \frac{d}{dt}\partial_t^{\rho-1} h(t).
$$
If in this definition we interchange the differentiation and
fractional integration, then we get the definition of a
regularized derivative, that is, the definition of a fractional
derivative in the sense of Caputo:
$$
D_t^\rho h(t)= \partial_t^{\rho-1}\frac{d}{dt} h(t).
$$
Note that if  $\rho=1$  , then fractional derivative coincides with the ordinary classical
derivative of the first order: 
$$
\partial_t h(t)= \frac{d}{dt} h(t).
$$

Let $\rho\in(0,1) $ be a fixed number and let $C((a,b); H)$ stand
for a set of continuous functions $u(t)$  of $t\in (a,b)$ with
values in $H$. Consider the Cauchy type problem:
\begin{equation}\label{prob1}
\left\{
\begin{aligned}
&\partial_t^\rho u(t) + Au(t) = 0,\quad 0<t\leq T;\\
&\lim\limits_{t\rightarrow 0}\partial_t^{\rho-1} u(t) = \varphi,
\end{aligned}
\right.
\end{equation}
where $\varphi$ is a given vector in $H$. If  $\rho=1$, then the initial condition has the form $u(0) = \varphi$.  This problem is called a
\emph{forward problem}.

\begin{defin}\label{def} A function $u(t)$ with the properties
$\partial_t^\rho u(t), Au(t)\in C((0,T]; H)$ and satisfying
conditions (\ref{prob1})  is called \textbf{the
 solution} of the forward problem (\ref{prob1}).
\end{defin}

Let us denote by $E_{\rho, \mu}(t)$ the Mittag-Leffler function of
the form
$$
E_{\rho, \mu}(t)= \sum\limits_{k=0}^\infty \frac{t^k}{\Gamma(\rho
k+\mu)}.
$$

We first prove the existence and uniqueness of a solution of
problem (\ref{prob1}).

\begin{thm}\label{tfp} For any $\varphi\in H$ problem (\ref{prob1}) has a unique solution
and this solution has the form
\begin{equation}\label{fp}
u(t)=\sum\limits_{k=1}^\infty t^{\rho -1}E_{\rho, \rho}
(-\lambda_k t^\rho)(\varphi, v_k)  v_k.
\end{equation}
\end{thm}
The problem  (\ref{prob1}) for various operators A has been considered by a number of authors. Let us mention only some of
these works. The case of one spatial variable $x\in \mathbb{R}$ and subdiffusion equation with $Au=u_{xx}$ considered, for example, in
the book of A.A. Kilbas et al. \cite{ash3} and monograph of A. V. Pskhu \cite{ash4}, and references in these works. The paper Gorenflo,
Luchko and Yamamoto \cite{ash5} is devoted to the study of subdiffusion equations in Sobelev spaces. In the paper by Kubica
and Yamamoto \cite{ash6}, initial-boundary value problems for equations with time-dependent coefficients are considered. In the
multidimensional case $(x\in R^N)$, instead of the differential expression $u_{xx}$, authors of the papers \cite{ash3}, \cite{ash7}-\cite{ash9} considered the Laplace operator and Umarov \cite{ash10} consedered pseudodifferential operators with constant coefficients in the whole
space  $\mathbb{R}^N$.

A result similar to the above, in the case when the fractional part of the equation (1) is the Caputo derivative, was obtained by M. Ruzhansky et al. \cite{ash11}.
In the case when $A$ is an arbitrary elliptic differential operator,
this theorem was proved in \cite{ash12}.

Obviously solution (\ref{fp}) depends on $\rho\in (0,1)$. Now let us consider the order of fractional derivative $\rho$ as a unknown
parameter and consider an inverse problem: can we identify
uniquely this parameter $\rho$, if we have as a additional
information the norm
\begin{equation}\label{exN}
W(t_0, \rho)= ||u(t_0)||^2=d_0
\end{equation}
at a fixed time instant $t_0>0$?

Problem (\ref{prob1}) together with extra condition (\ref{exN}) is
called \emph{the inverse problem}.

To solve this inverse problem we fix a number $\rho_0\in (0,1)$
and consider the problem for $\rho\in [\rho_0, 1)$.

\begin{defin}\label{Idef} A pair $\{u(t), \rho\}$ of the solution $u(t)$ to the forward problem
and the parameter  $\rho\in [\rho_0, 1]$ satisfying the
additional condition (\ref{exN}) is called the solution of the inverse problem.

\end{defin}

\begin{lem}\label{W} Given $\rho_0$ from interval $0<\rho_0< 1$, there exists
a number $T_0=T_0(\rho_0,\lambda_1)$, such that for all $t_0\geq
T_0$ and for arbitrary $\varphi\in H$ function $W(t_0, \rho)$ decreases monotonically with respect to $\rho\in [\rho_0,1]$.

\end{lem}

The main result of the paper is the following:

\begin{thm}\label{tin} Let $ \varphi \in H$ and $t_0\geq T_0$.
Then the inverse problem  has a unique solution $\{u(t), \rho \}$
if and only if
\[
W(t_0, 1)\leq d_0\leq W(t_0, \rho_0).
\]
\end{thm}

Theorem \ref{tin} gives a positive answer to the problem posed in
review article by Z. Li et al.  \cite{ash1} (p. 440) in the
Conclusions and Open Problems section: Is it possible to identify
uniquely the order of fractional derivatives if an additional
information about the solution is specified at a fixed time
instant as "the observation data"?.

Note that our result shows, that it is possible to restore the order of the fractional derivative by using the value of  $W(t,\rho)$ at a fixed time
instant $t_0$ as "the observation data".

The inverse problem of determining the order of time fractional derivative in subdiffusion equations has been studied
by a number of authors (see a survey paper \cite{ash1} and references therein, \cite{ash13}-\cite{ash23}). It is necessary to note that in all these
publications the following relation was taken as an additional condition
 
 \begin{equation}\label{ex1}
u(x_0,t)= h(t), \,\, 0<t<T,
\end{equation}
at a monitoring point $x_0\in \overline{\Omega}$.
But this condition, as a rule, can guarantee only the uniqueness of the solution of the inverse
problem (see \cite{ash13}-\cite{ash16}). However, as Theorem 2 states, unlike (\ref{ex1}), condition (\ref{exN}) guarantees both uniqueness
and the existence of a solution.

Hatano et al. \cite{ash17} considered  the equation $\partial_t^\rho
u =\triangle u$ with the Dirichlet boundary condition and the
initial function $\varphi(x)$ (see also \cite{ash18}). They proved
the following property of the parameter $\rho$: if $\varphi\in
C_0^\infty(\Omega)$ and $\triangle \varphi(x_0)\neq 0$, then
$$\rho =\lim\limits_{t\rightarrow
0}\big[t\partial_t u(x_0,t)[u(x_0,t)-\varphi(x_0)]^{-1}\big].$$

For the best of our knowledge, only in the paper \cite{ash19} by J.
Janno  the existence problem is considered. Giving an extra
boundary condition $B u(\cdot, t)= h(t), 0<t<T$ the author
succeeded to prove the existence theorem for determining the order
$\rho$, $0<\rho< 1$, of the Caputo derivative and the kernel of
the integral operator in the equation.

We also note the following recent papers. In the paper  by Z. Li
and Z. Zhang \cite{ash20} the authors studied the uniqueness
in an inverse problem for simultaneously determining the order of time fractional derivative and a source function in a
subdiffusion equation.
 In \cite{ash21}, 
M. Yamamoto proved the uniqueness in determining both orders of fractional time
derivatives and spatial derivatives in diffusion equations. The proof relies on the eigenfunction expansion and the
asymptotics of the Mittag-Leffler function. The authors of \cite{ash22} discuss similar issues discussed in the present paper. As
an additional information for inverse problem they have considered the value of projection of the solution onto the first
eigenfunction at a fixed time instance. Note, that results of paper \cite{ash22} are applicable only in case, when the first
eigenvalue of the corresponding elliptic operator is equal to zero. We also mention the paper \cite{ash23}, in which a result
similar to Theorem 2 was proved for the subdiffusion equation with the Caputo derivative. 
 
In conclusion, we give the following remarks:

1) As the operator $A$, one can take any equations of mathematical physics considered in Section 6 of the article by
M. Ruzhansky et al. \cite{ash11}, including the classical Sturm-Liouville problem, differential models with involution, fractional
Sturm-Liouville operators, harmonic and anharmonic oscillators, Landau Hamiltonians, fractional Laplacians, harmonic
and anharmonic operators on the Heisenberg group.

It should be noted, that the authors of \cite{ash11} considered inverse problems for restoring the right-hand side of a
subdiffusion equation for a large class of positive operators $A$.

2) Further, let us take $\mathbb{R}^N$ as a Hilbert space $H$ and
$N$-dimensional symmetric quadratic matrix $A=\{a_{i,j}\}$ with
constant elements $a_{i,j}$ as operator $A$. In this case, the
problem (\ref{prob1}) coincides with the Cauchy problem for a
linear system of fractional differential equations.

3) You can also consider various options for the function $W(t,
\rho)$. Examples $W(t, \rho)=||Au(t)||^2$, $W(t,
\rho)=(u,\varphi)$.

\section{Forward problem}
In the present section we prove Theorem \ref{tfp}.

To prove the existence of the forward problem's solution we remind
the following estimate of the Mittag-Leffler function with a
negative argument  (see, for example, \cite{ash24},  p. 136)
\begin{equation}\label{m1}
|E_{\rho, \mu}(-t)|\leq \frac{C}{1+ t}, \quad t>0.
\end{equation}

In accordance with Definition \ref{def}, we will first show that for function (\ref{fp}) one has $Au(t)\in C((0,T]; H)$. To do this,
consider the sum

$$
S_j(t)=\sum\limits_{k=1}^j t^{\rho-1}E_{\rho, \rho}(-\lambda_k
t^\rho)\, (\varphi,v_k)v_k.
$$
Then
\[
AS_j(t)=\sum\limits_{k=1}^j \lambda_k t^{\rho-1}
E_{\rho}(-\lambda_k t^\rho)\, (\varphi,v_k)v_k.
\]
Due to the Parseval equality we may write
$$
||AS_j(t)||^2=\sum\limits_{k=1}^j |\lambda_k t^{\rho-1}
E_{\rho}(-\lambda_k t^\rho)\, (\varphi,v_k)|^2 \leq C
t^{-2}||\varphi||^2.
$$
Here we used estimate (\ref{m1}) and the inequality $\lambda
t^{\rho}(1+\lambda t^{\rho})^{-1} < 1$.

Hence, we obtain  $Au(t)\in C((0,T]; H)$.

Further, from equation (\ref{prob1}) one has $\partial_t^\rho
S_j(t)= - AS_j(t)$. Therefore, from above reasoning, we finally
have $\partial_t^\rho u(t)\in C((0,T]; H)$.

It is not hard to verify the fulfillment of equation (\ref{prob1})
(see, for example, \cite{ash25}, p. 173 and \cite{ash26}) and the
initial condition therein.

Now we prove the uniqueness of the forward problem's solution.

Suppose that problem (\ref{prob1}) has two solutions $u_1(t)$ and
$u_2(t)$. Our aim is to prove that $u(t)=u_1(t)-u_2(t)\equiv 0$.
Since the problem is linear, then we have the following homogenous
problem for $u(t)$:
\begin{equation}\label{ur1}
\partial_t^\rho u(t) + Au(t) = 0, \quad t>0;
\end{equation}
\begin{equation}\label{nu1}
\lim\limits_{t\rightarrow 0}\partial_t^{\rho-1} u(t) = 0.
\end{equation}

Set
\[
w_k(t)\ =\ (u(t), v_k).
\]

It follows from (\ref{ur1}) that for any $k\in \mathbb{N}$
\[
\partial_t^\rho w_k(t) = (\partial_t^\rho u(t), v_k) = -\, (Au(t), v_k)= -\,
(u(t), Av_k) = -\, \lambda_k w_k(t).
\]

Therefore, we have the following Cauchy problem for $w_k(t)$ (see
(\ref{nu1})):
$$
\partial_t^\rho w_k(t) +\lambda_k w_k(t)=0,\quad t>0; \quad \lim\limits_{t\rightarrow 0}\partial_t^{\rho-1} w_k(t)=0.
$$
This problem has the unique solution (see, for example,
\cite{ash25}, p. 173 and \cite{ash26}). Therefore, $w_k(t) = 0$ for
$t>0$ and for all $k\geq 1$. Then by the Parseval equation we
obtain $u(t) = 0$ for all $t>0$. Hence uniqueness of the solution
is proved.

Thus the proof of Theorem \ref{tfp} is complete.

\section{Inverse problem}

\begin{lem}\label{ek}

Given $\rho_0$ from the interval $0<\rho_0< 1$, there exists a
number $T_0=T_0(\rho_0,\lambda_1)$, such that for all $t_0\geq
T_0$ and  $\lambda \geq \lambda_1$ functions $e_\lambda(\rho) =
t_0^{\rho-1} E_{\rho, \rho}(-\lambda t_0^\rho)$ are positive and they decrease monotonically with respect to $\rho\in [\rho_0, 1]$.

\end{lem}

\begin{proof} Let us denote by $\delta(1; \beta)$ a contour oriented by
non-decreasing $\arg \zeta$ consisting of the following parts: the
ray $\arg  \zeta = -\beta$, $|\zeta|\geq 1$, the arc $-\beta\leq
\arg \zeta \leq \beta$, $|\zeta|=1$, and the ray $\arg \zeta =
\beta$, $|\zeta|\geq 1$. If $0<\beta <\pi$, then the contour
$\delta(1; \beta)$ divides the complex $\zeta$-plane into two
unbounded parts, namely $G^{(-)}(1;\beta)$ to the left of
$\delta(1; \beta)$ by orientation, and $G^{(+)}(1;\beta)$ to the
right of it. The contour $\delta(1; \beta)$ is called the Hankel
path.

Let $\beta = \frac{3\pi}{4}\rho$, $\rho\in [\rho_0, 1)$. Then by
the definition of this contour $\delta(1; \beta)$, we arrive at
(see \cite{ash17}, formula (2.29), p. 135, note $-\lambda
t_0^\rho\in G^{(-)}(1;\beta)$)
\begin{equation}\label{Erho}
t_0^{\rho-1}E_{\rho, \rho}(-\lambda t_0^\rho)= -
\frac{1}{\lambda^2 t_0^{\rho+1} \Gamma (-\rho)}+\frac{\rho}{2\pi i
\lambda^2
t_0^{\rho+1}}\int\limits_{\delta(1;\beta)}\frac{e^{\zeta^{1/\rho}}\zeta^{\frac{1}{\rho}+1}}{\zeta+\lambda
t_0^\rho} d\zeta = f_1(\rho)+f_2(\rho).
\end{equation}

To prove the lemma, it is suffices to show that the derivative  $\frac{d}{d\rho} e_\lambda(\rho)$   is negative for all $\rho\in[\rho_0,1)$, since the positivity of $e_\lambda(\rho) $ follows from the inequality  $e_\lambda(1)=e^{-\lambda t}>0 $.

It is not hard to estimate the derivative
$f'_1(\rho)$. Indeed, let $\Psi(\rho)$ be the logarithmic derivative of the gamma
function $\Gamma(\rho)$ (for the definition and properties of
$\Psi$ see \cite{ash20}). Then $\Gamma'(\rho) = \Gamma (\rho)
\Psi(\rho)$, and therefore,
$$
f_1'(\rho)=\frac{\ln t_0 - \Psi (-\rho)}{\lambda^2 t_0^{\rho
+1}\Gamma (-\rho)}.
$$
Since
$$
\frac{1}{\Gamma(-\rho)}=-\frac{\rho}{\Gamma(1-\rho)}=-\frac{\rho(1-\rho)}{\Gamma(2-\rho)},
\quad \Psi(-\rho)=\Psi(1-\rho)
+\frac{1}{\rho}=\Psi(2-\rho)+\frac{1}{\rho}-\frac{1}{1-\rho},
$$
the function $f_1'(\rho)$  can be
represented as follows
\begin{equation}\label{f1}
f_1'(\rho)=\frac{1}{\lambda^2 t_0^{\rho+1}}
\frac{\rho(1-\rho)[\Psi (2-\rho)-\ln t_0]+1-2\rho}{ \Gamma
(2-\rho)}=-\frac{f_{11}(\rho)}{\lambda^2 t_0^{\rho+1}\Gamma
(2-\rho)}.
\end{equation}
If $\gamma\approx 0,57722$ is the Euler-Mascheroni constant, then
$\Psi(2-\rho)< 1-\gamma$. Therefore,
\[
f_{11}(\rho)>\rho (1-\rho)[\ln t_0 -(1-\gamma)])+2\rho-1.
\]
For $t_0=e^{1-\gamma} e^{2/\rho}$ one has $\rho (1-\rho)[\ln t_0
-(1-\gamma)])+2\rho-1=1$. Hence, $f_{11}(\rho)\geq 1$, provided
$t_0\geq T_0$ and
\begin{equation}\label{t0}
T_0= e^{1-\gamma} e^{2/\rho_0}.
\end{equation}
Thus, by virtue of (\ref{f1}), for all such $t_0$ we arrive at
\begin{equation}\label{f11}
f_1'(\rho)\leq-\frac{1}{\lambda^2 t_0^{\rho+1}}.
\end{equation}

To estimate the derivative   $f'_2(\rho)$, we denote the integrand in (\ref{Erho}) by $F(\zeta, \rho)$:

\[
F(\zeta, \rho)=\frac{1}{2\pi i \rho\lambda^2 t_0^{\rho+1}}\cdot
\frac{e^{\zeta^{1/\rho}}\zeta^{1/\rho+1}}{\zeta+\lambda
	t_0^\rho}.
\]

Note, that the domain of integration $\delta(1;\beta)$ also depends on $\rho$. To take this circumstance into account when differentiating
the function $f'_2(\rho)$, we rewrite the integral (\ref{Erho}) in the form:

\[
f_2(\rho)=f_{2+}(\rho)+f_{2-}(\rho)+f_{21}(\rho),
\]
where
\[
f_{2\pm}(\rho)=e^{\pm i \beta}\int\limits_1^\infty
F(s\,e^{\pm i\beta}, \rho)\, ds,
\]

\[
f_{21}(\rho) = i \int\limits_{-\beta}^{\beta} F(e^{i y},
\rho)\, e^{iy} dy= i\beta \int\limits_{-1}^{1} F(e^{i \beta s},
\rho)\, e^{i\beta s} ds.
\]

Let us consider the function  $f_{2+}(\rho)$. Since $\beta=\frac{3\pi}{4}\rho$
 and $\zeta= s\, e^{i\beta}$, then

\[
e^{\zeta^{1/\rho}}=e^{ \frac{1}{\sqrt{2}} (i-1)s^{\frac{1}{\rho}}},
\]

The derivative of the function $f_{2+}(\rho)$ has the form

$$
f_{2+}'(\rho)={I}\cdot\int\limits_1^\infty \frac{e^{ \frac{1}{\sqrt{2}} (i-1)s^{\frac{1}{\rho}}}s^{1/\rho+1}e^{2ia\rho}\big[\frac{1}{\rho^2}( \frac{1}{\sqrt{2}} (1-i)s^{\frac{1}{\rho}}-1)\ln s+2ia -\frac{1}{\rho}-\ln	t_0-\frac{ias e^{ia\rho} +\lambda
		t_0^\rho \ln t_0}{se^{ia\rho}  +\lambda	t_0^\rho}\big]}{se^{ia\rho}  +\lambda	t_0^\rho} ds.
$$
where $I=e^{ia}(2\pi i\rho\lambda^2t_0^{\rho+1})$ and $a=\frac{3\pi}{4}$. By virtue of the inequality  $|se^{ia \rho} + \lambda t_0^\rho| \geq \lambda t_0^\rho$
we arrive at

$$
|f_{2+}'(\rho)|\leq \frac{C}{\rho \lambda^3	t_0^{2\rho+1}} \cdot\int\limits_1^\infty e^{- \frac{1}{2} s^{\frac{1}{\rho}}}s^{1/\rho+1}\big[\frac{1}{\rho^2} s^{\frac{1}{\rho}}\ln s+\ln t_0 \big]ds.
$$

\begin{lem}\label{I} Let $0<\rho\leq 1$ and $m\in \mathbb{N}$.
Then
\[
I(\rho)=\frac{1}{\rho}\int\limits_1^\infty
e^{-\frac{1}{2}s^{\frac{1}{\rho}}} s^{\frac{m}{\rho}+1} ds\leq
C_m.
\]
\end{lem}
\begin{proof} Set $r=s^{\frac{1}{\rho}}$. Then
\[
s=r^\rho, \quad ds = \rho r^{\rho-1} dr.
\]
Therefore,
\[
I(\rho)=\int\limits_1^\infty e^{-\frac{1}{2}r} r^{m-1+2\rho}
dr\leq \int\limits_1^\infty e^{-\frac{1}{2}r} r^{m+1} dr = C_m.
\]

\end{proof}

Lemma \ref{I} is proved.

Application of this lemma gives (note, $ \frac{1}{\rho}ln s<s^{\frac{1}{\rho}}$, provided $ s \geq 1$)

$$
|f_{2+}'(\rho)|\leq \frac{C}{\lambda^3	t_0^{2\rho+1}} \big[\frac{C_3}{\rho}+C_1 ln t_0 \big] \leq \frac{C}{\lambda^3	t_0^{2\rho+1}}\big[\frac{1}{\rho}+ln t_0 \big].
$$

Function $f_{2-}(\rho)$ has exactly the same estimate.

Now consider the function $f_{21}(\rho)$. For its derivative we have

$$
f_{21}'(\rho)=\frac{a}{2\pi i\lambda^2 t_0^{\rho+1}}\cdot\int\limits_{-1}^1 \frac{e^{ e^{ias}} e^{ias}e^{2ia\rho s}\big[2ias -\ln	t_0-\frac{ias e^{ia\rho} +\lambda
		t_0^\rho \ln t_0}{e^{ia\rho s}  +\lambda	t_0^\rho}\big]}{e^{ia\rho s}  +\lambda	t_0^\rho} ds.
$$

Therefore,

\[
|f_{21}'(\rho)|\leq C
\frac{ln t_0}{\lambda^3 t_0^{2\rho+1}}.
\]

Taking into account stimmte  (\ref{f11}) and the estimates of $ f_{2\pm}' $ 
 and $ f_{21}' $, we have

\[
\frac{d}{d\rho} e_\lambda(\rho)< -\frac{1}{\lambda^2
	t_0^{\rho+1}}+C\frac{1/\rho+\ln t_0}{\lambda^3 t_0^{2\rho+1}}.
\]

In other words, this derivative is negative if
\[
t_0^{\rho_0}>\frac{C}{\lambda_0}(\frac{1}{\rho_0}+ \ln t_0).
\]

Hence, there exists a number $T_0=T_0(\lambda_0,\rho_0) $ (see also (\ref{t0})) such, that for all $t_0\geq T_0$

\[
\frac{d}{d\rho}[{t_0^{\rho-1}E_{\rho,\rho}(-\lambda t_0^{\rho})}] < 0, \quad \lambda\geq
\lambda_0,\quad \rho\in [\rho_0,1].
\]
Lemma \ref{ek} is proved.
\end{proof}
Since

$$
W(t,\rho)=||u(t)||^2=\sum\limits_{k=1}^\infty |(\varphi,
v_k)|^2|t^{\rho-1}E_\rho(-\lambda_k t^\rho)|^2,
$$
then Lemma \ref{W} follows immediately from Lemma \ref{ek}. Theorem \ref {tin}  is an easy consequence of these two lemmas.

\

\bibliographystyle{amsplain}

\end{document}